\documentclass[a4paper,11pt]{article}

\usepackage{a4wide}
\usepackage[english]{babel}
\usepackage{amsfonts}
\usepackage{amsmath}
\usepackage{graphicx}
\usepackage[colorlinks,bookmarks=true]{hyperref}
\usepackage{amsthm}
\usepackage{enumitem}

\newtheorem{lemma}{Lemma}[section]

\newtheorem{theorem}{Theorem}[section]
\newtheorem{proposition}{Proposition}[section]

\newtheorem{corollary}{Corollary}[section]

\newcommand{\E}{\mathbb{E}}
\newcommand{\PP}{\mathbb{P}}

\def\be{\begin{eqnarray}}
\def\ee{\end{eqnarray}}
\def\b*{\begin{eqnarray*}}
\def\eq*{\end{eqnarray*}}
\def\beq{\begin{equation}}
\def\eeq{\end{equation}}

\title{On the law of a triplet associated with the\\
pseudo-Brownian bridge}

\author{Mathieu Rosenbaum and Marc Yor\\$~~$\\
LPMA, University Pierre et Marie Curie (Paris 6)}

\begin{document}

\maketitle

\begin{abstract}
\noindent We identify the distribution of a natural triplet associated with the pseudo-Brownian bridge. In particular, for $B$ a Brownian motion and $T_1$ its first hitting time of the level one, this remarkable law allows us to understand some properties of the process $(B_{uT_1}/\sqrt{T_1},~u\leq 1)$ under uniform random sampling, a study started in \cite{elie2013expectation}.
\end{abstract}


\section{Introduction and main results}\label{intro}

Let $(B_t,~t\geq 0)$ be a standard Brownian motion and $(T_a,~a>0)$ its first hitting times process:
$$T_a=\text{inf}\{t,~B_t>a\}.$$
Let $U$ denote a uniform random variable on $[0,1]$, independent of $B$. In \cite{elie2013expectation}, the very interesting distribution of the random variable
$$\frac{B_{UT_a}}{\sqrt{T_a}}$$
is described. In particular, it is shown that this law does not depend on $a$, admits moments of any order and is centered. Furthermore, its density is quite remarkable, see \cite{elie2013expectation} for details. In this work, our goal is to extend this study by giving some general properties of the rescaled Brownian motion up to its first hitting time of level $1$, that is the process $(\alpha_u,~u\leq 1)$ defined by $$\alpha_u=\frac{B_{uT_1}}{\sqrt{T_1}}.$$ Here also, we will focus on the case of a uniform random sampling. Indeed, it is a very natural sampling scheme, which is known to lead to deep properties, see the seminal paper \cite{pitman1999brownian}.\\

\noindent Let $(L_t,~t\geq 0)$ be the local time of the Brownian motion at point $0$ and $(\tau_l,~l>0)$ be the inverse local time process:
$$\tau_l=\text{inf}\{t,~L_t>l\}.$$ It turns out that our results on the process $(\alpha_u)$ will be deduced from properties obtained on the pseudo-Brownian bridge. This process was introduced in \cite{biane1987processus} and is defined by
$$(\frac{B_{u\tau_1}}{\sqrt{\tau_1}},~u\leq 1).$$
The pseudo-Brownian bridge is equal to $0$ at time $0$ and time $1$ and has the same quadratic variation as the Brownian motion. Thus, it shares some similarities with the Brownian bridge, which explains its name. We refer to \cite{biane1987processus} for more on this process. More precisely, we consider in this paper the triplet
$$(\frac{B_{U\tau_1}}{\sqrt{\tau_1}},\frac{1}{\sqrt{\tau_1}},L_{U\tau_1}),$$ 
where $U$ is a uniform random variable on $[0,1]$, independent of $B$. Our main theorem, whose proof is given in Section \ref{proof} and where $\underset{\mathcal{L}}{=}$ denotes equality in law, is the following.

\begin{theorem}\label{theo1}
The following identity in law holds:
$$(\frac{B_{U\tau_1}}{\sqrt{\tau_1}},\frac{1}{\sqrt{\tau_1}},L_{U\tau_1})\underset{\mathcal{L}}{=}(\frac{1}{2}B_1,L_1,\Lambda),$$
with $\Lambda$ a uniform random variable on $[0,1]$, independent of $(B_1,L_1)$.
\end{theorem}

\noindent This result is quite remarkable in the sense that the distribution of the triplet is surprisingly simple. In particular, the marginal laws of the variables in the triplet are respectively Gaussian, absolute Gaussian and uniform distributions.\\

\noindent Let $(M_t,~t\geq 0)$ be the one sided supremum of $B$ and $(R_t,~t\geq 0)$ be a three dimensional Bessel process starting from $0$. We define the random variable $\gamma$ by $$\gamma=\text{sup}\{t\geq 0,~R_t=1\}$$ and the process $(J_u,~u\geq 0)$ by
$$J_u=\underset{t\geq u}{\text{inf}}R_t.$$ Using L\'evy's characterization of the reflecting Brownian motion and Pitman's representation of the three dimensional Bessel process, see \cite{pitman1975one} and for example \cite{revuz1999continuous}, the aficionados of Brownian motion will easily deduce the variants of Theorem \ref{theo1} stated in the following corollary, where $U$ is a uniform random variable independent of $B$ and $R$.

\begin{corollary}\label{cor1}
We have
\begin{align*}
(\frac{B_{UT_1}}{\sqrt{T_1}},\frac{1}{\sqrt{T_1}},M_{UT_1})\underset{\mathcal{L}}{=}(\Lambda L_1-\frac{1}{2}|B_1|,L_1,\Lambda),\\
(\frac{R_{U\gamma}}{\sqrt{\gamma}},\frac{1}{\sqrt{\gamma}},J_{U\gamma})\underset{\mathcal{L}}{=}(\Lambda L_1+\frac{1}{2}|B_1|,L_1,\Lambda),
\end{align*}
with $\Lambda$ a uniform random variable on $[0,1]$, independent of $(B_1,L_1).$
\end{corollary}

\noindent We clearly see that the distribution of the couple $(B_1,L_1)$ plays an essential role in the description of the laws of the triplets in Theorem \ref{theo1} and Corollary \ref{cor1}. Recall that for $s>0$, the density of $(B_s,L_s)$ at point $(x,l)\in\mathbb{R}\times\mathbb{R}^+$ is given by
\begin{equation}\label{dens}
\frac{1}{\sqrt{2\pi s^3}}(|x|+l)\text{exp}\big(-\frac{(l+|x|)^2}{2s}\big).
\end{equation}
From this expression, we deduce for fixed time $s\geq 0$ the useful factorization: 
$$(|B_s|,L_s)\underset{\mathcal{L}}{=}R_s(1-U,U).$$

\noindent Using elementary computations, this last expression enables to obtain interesting consequences of Theorem \ref{theo1} and Corollary \ref{cor1}. For example, we give in the following corollary unexpected equalities in law and independence properties for some Brownian functionals.

\begin{corollary}\label{cor2}
The three triplets
\begin{align*}
&\big(|B_{U\tau_1}|,L_{U\tau_1},\frac{1+2|B_{U\tau_1}|}{\sqrt{\tau_1}}\big),\\
&\big(M_{UT_1}-B_{UT_1},M_{UT_1},\frac{1+2(M_{UT_1}-B_{UT_1})}{\sqrt{T_1}}\big),\\
&\big(R_{U\gamma}-J_{U\gamma},J_{U\gamma},\frac{1+2(R_{U\gamma}-J_{U\gamma})}{\sqrt{\gamma}}\big)
\end{align*}
are identically distributed and their common law is that of the triplet of independent variables
$$\big(\frac{1}{2}(\frac{1}{U}-1),\Lambda,R_1\big).$$
\end{corollary}

\noindent As a general comment, we should mention that we in fact derived several ways to obtain the preceding results. We present here
the approach that we think is the most easily accessible.\\

\noindent The rest of the paper is organized as follows. Section \ref{proof} contains the proof of Theorem \ref{theo1}.
The law of the random variable $B_{UT_1}/\sqrt{T_1}$ is investigated in Section \ref{revisit} and some applications 
of Theorem \ref{theo1} and Corollary \ref{cor1} can be found in Section \ref{appli}. Some additional remarks are gathered in an appendix.

\section{Proof of Theorem \ref{theo1}}\label{proof}

\noindent In this section, we give the proof of our main result, Theorem \ref{theo1}. 

\subsection{Step 1: Introducing a Mellin type transform}

First remark that using the symmetry of the Brownian motion, Theorem \ref{theo1} is equivalent to the equality in law:
\begin{equation}\label{equiv}
(|\frac{B_{U\tau_1}}{\sqrt{\tau_1}}|,\frac{1}{\sqrt{\tau_1}},L_{U\tau_1})\underset{\mathcal{L}}{=}(\frac{1}{2}|B_1|,L_1,\Lambda).
\end{equation}
We now introduce the function $\mathcal{M}_l(a,c)$ which is defined for $a\geq 0$, $c\geq 0$ and $l\leq 1$ by 
$$\mathcal{M}_l(a,c)=\E\big[|\frac{B_{U\tau_1}}{\sqrt{\tau_1}}|^a|\frac{1}{\sqrt{\tau_1}}|^c\mathrm{1}_{\{L_{U\tau_1}\leq l\}}\big].$$
Using properties of the Mellin transform together with a monotone class argument, we easily get that proving \eqref{equiv} is equivalent to show the following equality:
\begin{equation}\label{equiv2}
\mathcal{M}_l(a,c)=l\E\big[|\frac{1}{2}B_1|^a|L_1|^c\big].
\end{equation}

\subsection{Step 2: The Mellin transform of $(B_1,L_1)$}\label{melbl}
In this second step, we give the Mellin transform of the couple $(B_1,L_1)$. Having an expression for this functional is of course helpful in order to prove Equality \eqref{equiv2}.
\subsubsection{The result}
The Mellin transform of the couple $(|B_1|,L_1)$ is given in the following proposition. 
\begin{proposition}\label{mel}
For $a>0$ and $c>0$, we have
$$\E[|B_1|^a(L_1)^c]=\frac{\Gamma(1+a)\Gamma(1+c)}{2^{\frac{a+c}{2}}\Gamma(1+\frac{a+c}{2})},$$
where $\Gamma$ denotes the classical gamma function.
\end{proposition}

\subsubsection{Technical Lemma}

Before starting the proof of Proposition \ref{mel}, we give a useful lemma.
\begin{lemma}\label{exp}
Let $\mathcal{E}$ and $\mathcal{E}'$ be two independent standard exponential variables, independent of $B$. The following equality in law holds: $$(|B_{2\mathcal{E}}|,L_{2\mathcal{E}})\underset{\mathcal{L}}{=}(\mathcal{E},\mathcal{E}').$$ 
\end{lemma}
\noindent This result is in fact quite well known. However, for sake of completeness we give its proof here.
\begin{proof}
To prove lemma \ref{exp}, we compute the Mellin transform of the couple $(|B_{2\mathcal{E}}|,L_{2\mathcal{E}})$. Let $a>0$ and $c>0$. From the law of the couple $(|B_t|,L_t)$ obtained from \eqref{dens}, we get
\begin{align*}
\E[|B_{2\mathcal{E}}|^a|L_{2\mathcal{E}}|^c]&=\frac{1}{2}\int_0^{+\infty}dt\text{e}^{-t/2}\E[|B_{t}|^a|L_{t}|^c]\\
&=\frac{1}{2}\int_0^{+\infty}dt\text{e}^{-t/2}\int_0^{+\infty}dx\int_0^{+\infty}dl
\sqrt{\frac{2}{\pi t^3}}x^al^c(x+l)\text{e}^{-\frac{(l+x)^2}{2t}}\\
&=\int_0^{+\infty}dx\int_0^{+\infty}dlx^al^c\int_0^{+\infty}\frac{dt}{\sqrt{2\pi t^3}}\text{e}^{-t/2}
(x+l)\text{e}^{-\frac{(l+x)^2}{2t}}.
\end{align*} 
Using the expression of the density of an inverse Gaussian random variable, it is easily seen that the integral in $(dt)$ is equal to $\text{exp}\big(-(x+l)\big)$. Therefore, the remaining double integral is equal to $\Gamma(1+a)\Gamma(1+c)$. This gives the result since
$$\E[|\mathcal{E}|^a|\mathcal{E}'|^c]=\Gamma(1+a)\Gamma(1+c).$$
\end{proof}

\subsubsection{Proof of Proposition \ref{mel}}
We now give the proof of Proposition \ref{mel}. First remark that, by scaling, we have
$$\sqrt{2\mathcal{E}}(|B_1|,L_1)\underset{\mathcal{L}}{=}(|B_{2\mathcal{E}}|,L_{2\mathcal{E}}).$$
Thus, using Lemma \ref{exp}, we get
$$\E[(2\mathcal{E})^{\frac{a+c}{2}}]\E[|B_1|^aL_1^c]=\E[|\mathcal{E}|^a|\mathcal{E}'|^c]=\Gamma(1+a)\Gamma(1+c).$$
We eventually obtain the result since $$\E[(2\mathcal{E})^{\frac{a+c}{2}}]=2^{\frac{a+c}{2}}\Gamma(1+\frac{a+c}{2}).$$ 

\subsection{Step 3: End of the proof of Theorem \ref{theo1}}

According to Proposition \ref{mel}, in order to prove Theorem \ref{theo1}, it suffices to show that
\begin{equation}\label{identity}
\mathcal{M}_l(a,c)=l\frac{\Gamma(1+a)\Gamma(1+c)}{2^{\frac{a+c}{2}}\Gamma(1+\frac{a+c}{2})2^a}.
\end{equation}
To this purpose, let us write $\mathcal{M}_l(a,c)$ under the form
\begin{align*}
\mathcal{M}_l(a,c)&=\E\big[\frac{1}{\tau_1^{\frac{a+c}{2}+1}}\int_0^{\tau_1}ds|B_s|^a\mathrm{1}_{\{L_s\leq l\}}\big]\\
&=\E\big[\frac{1}{\tau_1^{\frac{a+c}{2}+1}}\int_0^{\tau_l}ds|B_s|^a\big].
\end{align*}
Then, using the definition of the Gamma function and a change of variable, we easily get
\begin{equation}\label{defmel}
\mathcal{M}_l(a,c)=\frac{1}{2^{\frac{a+c}{2}}\Gamma(1+\frac{a+c}{2})}\int_0^{+\infty}d\mu\mu^{1+a+c}\E\big[\int_0^{\tau_l}ds|B_s|^a\text{e}^{-\frac{\mu^2}{2}\tau_1}\big].
\end{equation}
\noindent Now recall that the process $\tau_l$ is a subordinator with Laplace exponent at point $\lambda\in\mathbb{R}^{+*}$ equal to $\sqrt{2\lambda}$, see \cite{revuz1999continuous}. Therefore,
\begin{align*}
\E\big[\int_0^{\tau_l}ds|B_s|^a\text{e}^{-\frac{\mu^2}{2}\tau_1}\big]&=\E\big[\int_0^{\tau_l}ds|B_s|^a\text{e}^{-\frac{\mu^2}{2}\tau_l-\mu(1-l)}\big]\\
&=\text{e}^{-\mu}\E\big[\int_0^{\tau_l}ds|B_s|^a\text{e}^{-\frac{\mu^2}{2}\tau_l-\mu(|B_{\tau_l}|-L_{\tau_l})}\big].
\end{align*}
Then, using L\'evy's theorem, see \cite{revuz1999continuous}, together with the optional stopping theorem, we get

$$
\E\big[\int_0^{\tau_l}ds|B_s|^a\text{e}^{-\frac{\mu^2}{2}\tau_1}\big]=\text{e}^{-\mu}\E\big[\int_0^{+\infty}ds|B_s|^a\text{e}^{\mu(L_s-|B_s|)-\frac{\mu^2s}{2}}\mathrm{1}_{\{L_s\leq l\}}\big].
$$
With the help of the joint law of $(|B_s|,L_s)$ given in \eqref{dens}, we obtain that this last quantity is also equal to
$$\text{e}^{-\mu}\int_0^{+\infty}dx\int_0^ldmx^a\text{e}^{\mu(m-x)}2\int_0^{+\infty}\frac{ds}{\sqrt{2\pi s^3}}(m+x)\text{e}^{-\frac{(m+x)^2}{2s}}\text{e}^{-\frac{\mu^2 s}{2}}.$$
Using again the expression of the density of an inverse Gaussian random variable, we see that the integral in $(ds)$ is equal to $\text{exp}\big(-\mu(x+m)\big)$.
Consequently, we obtain
\begin{align*}
\E\big[\int_0^{\tau_l}ds|B_s|^a\text{e}^{-\frac{\mu^2}{2}\tau_1}\big]&=2\text{e}^{-\mu}\int_0^{+\infty}dx\int_0^ldmx^a\text{e}^{-2\mu x}\\
&=2l\text{e}^{-\mu}\frac{\Gamma(1+a)}{(2\mu)^{a+1}}.
\end{align*}
Plugging this equality into Equation \eqref{defmel} gives
$$\mathcal{M}_l(a,c)=\frac{l}{2^{\frac{a+c}{2}}\Gamma(1+\frac{a+c}{2})}\frac{\Gamma(1+a)\Gamma(1+c)}{2^a},$$
which is the desired identity \eqref{identity}.

\section{The properties of the law of $\alpha$ revisited}\label{revisit}
In this section, we focus on the random variable
$$\alpha=\frac{B_{UT_1}}{\sqrt{T_1}},$$ 
with $U$ a uniform random variable on $[0,1]$ independent of $B$.
\subsection{Two equivalent characterizations of the distribution of $\alpha$}
In \cite{elie2013expectation}, we describe in term of its density the law of the variable 
$\alpha.$ In particular, we show that this variable is centered. This characterization of the distribution of $\alpha$ is obtained thanks to the computation of the Mellin transform of the positive and negative parts of $\alpha$. More precisely, we have for $m>0$
\begin{align*}
\E[(\alpha_+)^m]&=\E[|N|^m]2\int_0^1dz\frac{z^{1+m}}{(1+2z)}\\
\E[(\alpha_-)^m]&=\E[|\frac{N}{2}|^m]\big(\frac{\text{log}(3)}{2}\big).
\end{align*}
This leads to the following description of the law of $\alpha$, which is of course equivalent to that given in \cite{elie2013expectation}.
\begin{proposition}\label{desc}
Let $N$ and $Z$ be two independent random variables with $N$ standard Gaussian and $Z$ with density
$$\frac{2z}{(1-\frac{\emph{log}(3)}{2})(1+2z)}\mathrm{1}_{\{0<z<1\}}.$$
We have the following equalities in law:
\begin{align*}
&(\alpha|\alpha>0)\underset{\mathcal{L}}{=}|N|Z\\
&(-\alpha|\alpha<0)\underset{\mathcal{L}}{=}\frac{1}{2}|N|\\
&\PP[\alpha>0]=1-\frac{\emph{log}(3)}{2}.
\end{align*}
\end{proposition}

\noindent We wish to compare this characterization with the one obtained from Corollary \ref{cor1}, namely
\begin{equation}\label{desc2}
\alpha\underset{\mathcal{L}}{=}\Lambda L_1-\frac{1}{2}|B_1|.
\end{equation}
We first note that the centering property is easily recovered since
$$\E[\Lambda L_1-\frac{1}{2}|B_1|]=\frac{1}{2}\E[L_1-|B_1|]=0.$$
Moreover, the second moment can also be computed without difficulty. Indeed,
$$
\E[(\Lambda L_1-\frac{1}{2}|B_1|)^2]=\frac{1}{3}\E[L_1^2]-\frac{1}{2}\E[L_1|B_1|]+\frac{1}{4}.
$$
Then, using for example the formula for the Mellin transform of the couple $(|B_1|,L_1)$ given in Proposition \ref{mel}, we get
$$\E[(\Lambda L_1-\frac{1}{2}|B_1|)^2]=\frac{1}{3}.
$$
We now show that Proposition \ref{desc} and Equation \eqref{desc2} match. This is deduced from the following elementary description
of the random variable $A$ defined by
$$A=\Lambda U-\frac{1}{2}(1-U),$$
with $\Lambda$ and $U$ two independent uniform random variables on $[0,1]$.
\begin{proposition}\label{desc3}
Let $Z$ be the random variable defined in Proposition \ref{desc} and $V$ a uniform variable on $[0,1]$, independent of $Z$. We have
\begin{align*}
&(A|A>0)\underset{\mathcal{L}}{=}VZ\\
&(-A|A<0)\underset{\mathcal{L}}{=}\frac{1}{2}V\\
&\PP[A>0]=1-\frac{\emph{log}(3)}{2}.
\end{align*}
\end{proposition}
\begin{proof}
Let $f$ be a positive measurable function. The density of $A$ can be computed directly as follows. We have
\begin{align*}
\E[f(A)]&=\int_0^1d\lambda\E\big[f\big(\lambda U-\frac{1}{2}(1-U)\big)\big]\\
&=\int_{\mathbb{R}}dxf(x)u(x),
\end{align*}
with
$$u(x)=\E\big[\frac{1}{U}\mathrm{1}_{\{\frac{1+2x}{3}\leq U\leq 1+2x\}}\big].$$
Now, note that on the one hand
\begin{itemize}
\item[-] if $-\frac{1}{2}\leq x\leq 0$, 
$$u(x)=\int_{\frac{1+2x}{3}}^{1+2x}\frac{du}{u}=\text{log}(3),$$
\item[-] if $0\leq x\leq 1$,  
$$u(x)=\int_{\frac{1+2x}{3}}^{1}\frac{du}{u}=\text{log}(\frac{3}{1+2x}),$$
\item[-] if $x\leq -\frac{1}{2}$ or $x>1$,
$$u(x)=0.$$
\end{itemize}
On the other hand, it is easily seen that the density of $VZ$ is
$$\frac{1}{1-\frac{\text{log}(3)}{2}}\text{log}\big(\frac{3}{1+2x}\big).$$
This ends the proof of Proposition \ref{desc3}.
\end{proof}

\noindent Let us now start from Equation \eqref{desc2} and recover Proposition \ref{desc}. From Equation \eqref{desc2}, with our usual notation, we get
$$\alpha\underset{\mathcal{L}}{=}R_1A.$$
Thus, $$(\alpha|\alpha>0)\underset{\mathcal{L}}{=}(R_1A|A>0).$$
From Proposition \ref{desc3}, we obtain
$$(R_1A|A>0)\underset{\mathcal{L}}{=}R_1VZ\underset{\mathcal{L}}{=}|N|Z,$$
which gives the first result in Proposition \ref{desc}. The two other results are proved similarly, with the help of Proposition \ref{desc3}.

\subsection{A warning and some developments around Proposition \ref{desc}}
Since $1/\sqrt{T_1}$ is distributed as $|N|$, from Proposition \ref{desc}, it may be tempting to think that $$(B_{UT_1}|B_{UT_1}>0)$$ is distributed as $Z$ and is independent of $T_1$. However, this is wrong. This incited us to look at the joint law of $1/\sqrt{T_1}$ and $B_{UT_1}$. Indeed, although encoded in Corollary \ref{cor1}, it may deserve an explicit presentation which we give in the following theorem.
\begin{theorem}\label{descB}
$~~$
\begin{itemize}
\item Let $p\geq 0$. For $\phi$ a positive measurable function, the following formulas hold:
\begin{align*}
\E\big[\frac{1}{(\sqrt{T_1})^p}\phi(B_{UT_1})\mathrm{1}_{\{B_{UT_1}>0\}}\big]&=c_p\int_0^1db\phi(b)\big(1-\frac{1}{(3-2b)^{p+1}}\big)\\
\E\big[\frac{1}{(\sqrt{T_1})^p}\phi(B_{UT_1})\mathrm{1}_{\{B_{UT_1}<0\}}\big]&=c_p\int_{-\infty}^0dx\phi(x)\big(\frac{1}{(1-2x)^{p+1}}-\frac{1}{(3-2x)^{p+1}}\big),
\end{align*}
where $c_p=\E[|N|^p]=\frac{\Gamma(1+p)}{2^{p/2}\Gamma(1+p/2)}.$
\item The joint law of $(1/\sqrt{T_1},B_{UT_1})$ admits the density $h$ defined on $\mathbb{R}^{+*}\times(-\infty,1]$ by
\begin{align*}h(z,x)&=\sqrt{\frac{2}{\pi}}\big(\emph{e}^{-z^2/2}-\emph{e}^{-(3-2x)^2z^2/2}\big)\mathrm{1}_{\{z>0,~0<x<1\}}\\
&+\sqrt{\frac{2}{\pi}}\big(\emph{e}^{-(1-2x)^2z^2/2}-\emph{e}^{-(3-2x)^2z^2/2}\big)\mathrm{1}_{\{z>0,~x<0\}}.
\end{align*}
\item The law of $B_{UT_1}$ admits the density $k$ defined on $(-\infty,1)$ by
$$k(x)=\frac{2(1-x)}{3-2x}\mathrm{1}_{\{0<x<1\}}+\frac{2}{(1-2x)(3-2x)}\mathrm{1}_{\{x<0\}}.$$
\end{itemize}
\end{theorem}
\noindent Remark that from Theorem \ref{descB}, we get that 
$$(1-B_{UT_1}|B_{UT_1}>0)$$ is distributed as $Z$. We now give the proof of Theorem \ref{descB}.
\begin{proof}
To prove the first part of Theorem \ref{descB}, we use the fact that
$$\E\big[\frac{1}{(\sqrt{T_1})^p}\phi(B_{UT_1})\mathrm{1}_{\{B_{UT_1}>0\}}\big]$$
is equal to
$$\frac{1}{2^{p/2}\Gamma(1+p/2)}\int_0^{+\infty}d\mu\mu^{1+p}\E\big[\int_0^{T_1}ds\phi(B_s)\mathrm{1}_{\{B_s>0\}}\text{e}^{-\frac{\mu^2}{2}T_1}\big].$$
From Proposition 3.1 in \cite{elie2013expectation} (or the computation of $I_{\mu}$ in the same paper), the above expectation is equal to
$$\int_0^1db\phi(b)\frac{1}{\mu}(\text{e}^{-\mu}-\text{e}^{-\mu(3-2b)}).$$
Then, using Fubini's theorem and integrating in $\mu$ yields the first formula. The second formula is obtained in the same manner.\\

\noindent For the proof of Part 2 in Theorem \ref{descB}, in order to obtain the density at point $(z,x)$ in $\mathbb{R}^{+*}\times(0,1)$, $h(z,x)$, we note that from
Part 1 of Theorem \ref{descB}, the formula
$$\E\big[f(\frac{1}{\sqrt{T_1}},B_{UT_1})\mathrm{1}_{\{B_{UT_1}>0\}}\big]=\E[f(|N|,V)]-\int_0^1\frac{db}{3-2b}\E\big[f\big(\frac{|N|}{3-2b},b\big)\big],$$
with $V$ a uniform variable on $[0,1]$ independent of $N$, holds for every function $f$ of the form
$$f(z,b)=z^p\phi(b),$$ with $\phi$ some positive measurable function. An application of the monotone class theorem yields the validity of
the above formula for every positive measurable function $f$. Then, a simple change of variables gives the first formula in Part 2. The second formula is proved likewise.\\

\noindent To prove Part 3 in Theorem \ref{descB}, it suffices to take $p=0$ in Part 1.

\end{proof}

\section{Applications}\label{appli}
\noindent In this section, we give two applications of Theorem \ref{theo1} and Corollary \ref{cor1}.

\subsection{A family of centered Brownian functionals}

In \cite{elie2013expectation}, we established that the variable $H$ defined by
$$H=\frac{1}{T_1^{3/2}}\int_0^{T_1}ds B_s$$
admits moments of all orders and is centered. This centering property is equivalent to that of the random variable
$$\alpha=\frac{B_{UT_1}}{\sqrt{T_1}},$$
where $U$ is a uniform random variable on $[0,1]$, independent of $B$. In fact, Theorem \ref{theo1} and Corollary \ref{cor1} enable to build families
of centered functionals involving the Brownian motion and its first hitting time of level $1$, the local time at point $0$ and its inverse process, and the running maximum.
We have the following theorem, in which $H_1$ is equal to $H$.

\begin{theorem}\label{centeredva} 
For any $p\geq 1$, the random variables $H_p$ and $H'_p$ are centered, with
$$H_p=\frac{1}{T_1^{p/2+1}}\int_0^{T_1}ds\big((\frac{p+1}{2p^2}-1)M_s^p+B_sM_s^{p-1}\big)$$ and
$$H'_p=\frac{1}{\tau_1^{p/2+1}}\int_0^{\tau_1}ds\big(\frac{p+1}{2p^2}L_s^p-|B_s|L_s^{p-1}\big).$$
\end{theorem} 
\begin{proof}
First remark that from L\'evy's equivalence theorem, $H_p$ and $H'_p$ have the same law. Then, obviously the expectation of $H'_p$ is that of
$$(\frac{p+1}{2p^2})\frac{L_{U\tau_1}^p}{(\sqrt{\tau_1})^p}-\frac{|B_{U\tau_1}|(L_{U\tau_1})^{p-1}}{(\sqrt{\tau_1})^p}.$$
From Theorem \ref{theo1}, this random variable has the same law as
$$(\frac{p+1}{2p^2})\Lambda^pL_1^p-\frac{1}{2}|B_1|\Lambda^{p-1}L_1^{p-1}.$$
The expectation of this last quantity is equal to
$$\frac{1}{2p^2}\E[L_1^p]-\frac{1}{2p}\E[|B_1|L_1^{p-1}].$$
Using for example the Mellin transform of the couple $(|B_1|,L_1)$ given in Proposition \ref{mel}, it is easily seen that this expression is equal to zero (one may also use the martingale property of $\frac{1}{p}L_t^p-|B_t|L_t^{p-1}$, see \cite{revuz1999continuous}).
\end{proof}

\subsection{On the law of $R_{U\gamma}/\sqrt{\gamma}$}
We now focus on the distribution of $R_{U\gamma}/\sqrt{\gamma}$. From Corollary \ref{cor1}, we know that
$$\frac{R_{U\gamma}}{\sqrt{\gamma}}\underset{\mathcal{L}}{=}\Lambda L_1+\frac{1}{2}|B_1|\underset{\mathcal{L}}{=}R_1A',$$ 
with
$$A'=\Lambda U+\frac{1}{2}(1-U).$$
There is the following description of the laws of $A'$ and $R_{U\gamma}/\sqrt{\gamma}$.

\begin{proposition}\label{bes}
The law of $A'$ admits the density $l$ given by
$$l(a)=\emph{log}\big(\frac{1}{|2a-1|}\big)\mathrm{1}_{\{0<a<1\}}.$$
Consequently, the law of $R_{U\gamma}/\sqrt{\gamma}$ admits the following density:
$$\sqrt{\frac{2}{\pi}}x^2\int_1^{+\infty}dy y \emph{exp}(-\frac{x^2y^2}{2})l(\frac{1}{y}).$$
\end{proposition}
\begin{proof}
The density of $A'$ is obtained thanks to straightforward computations. Let $f$ be a positive measurable function. Using the density of a three dimensional Bessel variable $R_1$, see \cite{revuz1999continuous}, we have
$$\E[f(R_1A')]=\sqrt{\frac{2}{\pi}}\int_0^{+\infty}dr r^2\text{e}^{-r^2/2}\E[f(rA')].$$
Now, using the density of $A'$, we get
\begin{align*}
\E[f(R_1A')]&=\sqrt{\frac{2}{\pi}}\int_0^{+\infty}dr r^2\text{e}^{-r^2/2}\int_0^r\frac{dx}{r}l\big(\frac{x}{r}\big)f(x)\\
&=\int_0^{+\infty}dxf(x)\sqrt{\frac{2}{\pi}}\int_x^{+\infty}dr r\text{e}^{-r^2/2}l\big(\frac{x}{r}\big)\\
&=\int_0^{+\infty}dxf(x)\sqrt{\frac{2}{\pi}}x^2\int_1^{+\infty}dy y\text{e}^{-x^2y^2/2}l\big(\frac{1}{y}\big).
\end{align*}
\end{proof}

\section{Conclusion and future work}

In this paper, we establish the law of a triplet associated with the pseudo-Brownian bridge. This process has been introduced in \cite{biane1987processus} and is defined as
$$\big(\frac{B_{u\tau_1}}{\sqrt{\tau_1}},~u\leq 1\big).$$
In particular, this enables us to understand in depth some properties of the random variable
$$\alpha=\frac{B_{UT_1}}{\sqrt{T_1}}$$
studied in \cite{elie2013expectation}. In a forthcoming work, we intend to develop some consequences of the obtained results for the Brownian bridge and the Brownian meander.

\newpage
\appendix

\section*{Appendices}

\section{A simple proof for the joint law of $(1/\sqrt{\tau_1},L_{U\tau_1})$}
The fact that
$$(\frac{1}{\sqrt{\tau_1}},L_{U\tau_1})\underset{\mathcal{L}}{=}(L_1,\Lambda)$$
can obviously be deduced from Theorem \ref{theo1}. However, interestingly, we can give a simple proof for this equality in law.
Indeed, for $\lambda\geq 0$ and $l<1$, we have
\begin{align*}
\E[\text{e}^{-\lambda\tau_1}\mathrm{1}_{\{L_{U\tau_1}\leq l\}}]&=\E[\frac{1}{\tau_1}\int_0^{\tau_1}ds\mathrm{1}_{\{L_{s}\leq l\}}\text{e}^{-\lambda\tau_1}]\\
&=\E[\frac{\tau_l}{\tau_1}\text{e}^{-\lambda\tau_1}].
\end{align*}
Now, consider in general $(\tau_l)$ a subordinator and denote by $\psi$ its Laplace exponent.
Thus, we have
\begin{align*}
\E[\frac{\tau_l}{\tau_1}\text{e}^{-\lambda\tau_1}]&=\E[\tau_l\int_0^{+\infty}dt\text{e}^{-(t+\lambda)\tau_1}]\\
&=\int_0^{+\infty}dt\E[\tau_l\text{e}^{-(t+\lambda)\tau_l}]\text{e}^{-(1-l)\psi(t+\lambda)}.
\end{align*}
Using the fact that the Laplace exponent is differentiable on $\mathbb{R}^{+*}$, we get
\begin{align*}
\E[\frac{\tau_l}{\tau_1}\text{e}^{-\lambda\tau_1}]&=\int_0^{+\infty}dt l\psi'(t+\lambda)\text{e}^{-l\psi(t+\lambda)}\text{e}^{-(1-l)\psi(t+\lambda)}\\
&=l\text{e}^{-\psi(\lambda)}.
\end{align*}
This proves the independence of $1/\sqrt{\tau_1}$ and $L_{U\tau_1}$ and the fact that 
$L_{U\tau_1}$ is uniformly distributed. The equality in law
$$\frac{1}{\sqrt{\tau_1}}\underset{\mathcal{L}}{=}L_1$$
is easily obtained by scaling.
\newpage
\section{On a one parameter family of random variables including $\alpha$}

In this section of the appendix, we consider the family of variables defined for $0<c\leq 1$ by
$$\alpha_c=\Lambda L_1-c|B_1|,$$
as an extension of our study of
$$\alpha\underset{\mathcal{L}}{=}\alpha_{1/2}\underset{\mathcal{L}}{=}B_{UT_1}/\sqrt{T_1}.$$
The variables $\alpha_c$, although less natural than $\alpha_{1/2}$, enjoy some similar remarkable properties. Indeed, Proposition \ref{desc3} and Proposition \ref{desc} admit the following extensions.
\begin{proposition}\label{desc4}
Let $0<c\leq 1$ and $C=1/c$. Let $\Lambda$ and $U$ be two independent uniform variables on $[0,1]$ and 
$$A_c=\Lambda U-c(1-U).$$ We have
\begin{align*}
&(A_c|A_c>0)\underset{\mathcal{L}}{=}VZ_C\\
&(-A_c|A_c<0)\underset{\mathcal{L}}{=}cV\\
&\PP[A_c>0]=1-c\emph{log}(1+C),
\end{align*}
where $V$ and $Z_C$ are independent, with $V$ uniform on $[0,1]$ and $Z_C$ a random variable with density given by
$$\frac{C}{1-c\emph{log}(1+C)}\frac{dz z}{(1+Cz)}\mathrm{1}_{\{0<z<1\}}.$$
\end{proposition}

\begin{proposition}\label{desc5}
Let $0<c\leq 1$. The following equalities in law hold.
\begin{align*}
&(\alpha_c|\alpha_c>0)\underset{\mathcal{L}}{=}|N|Z_C\\
&(-\alpha_c|\alpha_c<0)\underset{\mathcal{L}}{=}c|N|\\
&\PP[\alpha_c>0]=1-c\emph{log}(1+C).
\end{align*}
\end{proposition}
\begin{proof}
To establish Proposition \ref{desc4}, we simply compute the density of $A_c$.
Proposition \ref{desc5} ensues since
$$\alpha_c\underset{\mathcal{L}}{=}R_1A_c$$ and
$$R_1V\underset{\mathcal{L}}{=}|N|,$$
with the same notation as previously.
\end{proof}
\newpage
\bibliographystyle{abbrv}
\bibliography{bibli_ry}

\end{document}